\documentclass[a4paper]{amsart}%
\usepackage{amsmath}
\usepackage{amssymb}
\usepackage{amsfonts}
\usepackage[T1]{fontenc}
\usepackage{cite}
\usepackage{graphicx}%
 \makeatletter \theoremstyle{plain}

\numberwithin{equation}{section}
\numberwithin{figure}{section}
\theoremstyle{plain}

\theoremstyle{plain}

\theoremstyle{plain}

\theoremstyle{plain}

\theoremstyle{plain}

\newtheorem{theorem}{Theorem}
\theoremstyle{plain}

\newtheorem{corollary}{Corollary}

\newtheorem{lemma}{Lemma}

\newtheorem{problem}{Problem}
\newtheorem{proposition}{Proposition}

\numberwithin{equation}{section}
\begin{document}
\title[On Some Problems of Operator Theory]{On Some Problems of Operator Theory and Complex Analysis}
\author{Mubariz T. Garayev}
\address{Institute of Mathematics and Mechanics, Ministry of Science and Education
Republic of Azerbaijan, B. Vagabzade str. 9, Baku 370141, Azerbaijan.}
\curraddr{Mathematics \ Department, College of Science, King Saud University , P.O.Box
2455, Riyadh 11451, Saudi Arabia}
\email{mgarayev@ksu.edu.sa}

\begin{abstract}
In 1955, Kadison \cite{14} asked whether the analogue of the classical
Burnside's theorem of the Linear Algebra holds in the infinite dimensional
case. We use reproducing kernels method to solve the Kadison question. Namely,
we prove that any proper weakly closed subalgebra $\mathcal{A}$ of the algebra
$\mathcal{B}\left(  H\right)  $ of bounded linear operators on infinite
dimensional complex Hilbert spaces $H$ has a nontrivial invariant subspace,
i.e., $\mathcal{A}$ is a nontransitive algebra. This solves The Transitive
Algebra Problem positively, and hence Hyperinvariant Subspace Problem and
Invariant Subspace Problem are also solved positively. In this context, we
also consider the celebrated Riemann Hypothesis of the theory of meromorphic
functions and solve it in negative.

\end{abstract}
\maketitle

\section{Introduction}

Let $H$ be an infinite dimensional complex Hilbert space, and $\mathcal{B}%
\left(  H\right)  $ be a Banach algebra of all bounded linear operators on
$H.$ A closed subspace $E$ of $H$ is said to be nontrivial invariant subspace
of an operator $T$ in $\mathcal{B}\left(  H\right)  $ if $\left\{  0\right\}
\neq E\neq H$ and $TE\subset E$, i.e., for each $x\in E$, $Tx\in E.$ It is
called hyperinvariant if $AE\subset E$ for every operator $A$ in the commutant
$\left\{  T\right\}  ^{^{\prime}}=\left\{  S\in\mathcal{B}\left(  H\right)
:ST=TS\right\}  $ of operator $T.$

The following two famous questions of operator theory and functional analysis
are open \cite{CP2011, 38, 28}:

\begin{problem}
\label{P1}(Invariant Subspace Problem). Does every bounded linear operator on
a Hilbert space $H$ have a nontrivial invariant subspace?
\end{problem}

\begin{problem}
\label{P2}(Hyperinvariant Subspace Problem). Does every nonscalar bounded
linear operator on a Hilbert space $H$ have a proper hyperinvariant subspace?
\end{problem}

Despite a number of partial results in direction of solving the invariant and
hyperinvariant subspace problems, these questions have remained open.

Obviously, a scalar operator $T$, i.e., an operator of the form $T=cI$,
$c\in\mathbb{C}$, where $I$ is an identity operator on $H$, has not a
nontrivial hyperinvariant subspace since $\left\{  T\right\}  ^{\prime
}=\mathcal{B}\left(  H\right)  $ and $\mathcal{B}\left(  H\right)  $ is a
transitive algebra.

Through the paper, the term algebra will be used to mean a weakly closed
subalgebra containing the identity of the Banach algebra $\mathcal{B}\left(
H\right)  .$ $A$ subalgebra $\mathcal{A}$ of $\mathcal{B}\left(  H\right)  $
is said to be transitive if it has no nontrivial invariant subspace. The
terminology comes from the fact, which is well-known from the of famous
Lomonosov's Lemma \cite{18} (see also Arveson \cite{A1967}), that
$\mathcal{A}$ is transitive if and only if $\mathcal{A}x$ is dense in $H$ for
every nonzero $x$ in $H.$ It is well known (and easy to prove) that
$\mathcal{B}\left(  H\right)  $ is a transitive algebra. It is not known: is
$\mathcal{B}\left(  H\right)  $ the only transitive algebra?

So, the transitive algebra problem raised by Kadison in his paper \cite{14} is
the following problem.

\begin{problem}
\label{P3}(The Transitive Algebra Problem). $\mathcal{A}$ is any transitive
algebra, is $\mathcal{A}=\mathcal{B}\left(  H\right)  $?\newline
\end{problem}

In case of Banach space the invariant subspace problem has been answered in
the negative by Per Enflo in his seminal paper \cite{10}. We note that in his
recent arXiv's preprint \cite{11} Enflo asserts that he positively solves the
invariant subspace problem (see Problem \ref{P1} above) in Hilbert spaces,
while nowadays the official expert's decision about validity of the result,
apparently, absent. In any case, in the present article, in Section 2, we
positively solve a more general question, namely, we give a positive answer to
the above mentioned Problem \ref{P3}. Our result obviously implies positive
solutions of Problems \ref{P1} and \ref{P2}.

For the history, known results and some recent developments on the invariant
subspace problem and hyperinvariant subspace problem, see, for instance, the
works \cite{CP2011, B1988, 29, CG2016, 10, 23, 38, 28, 35} and their references.

In this paper, using the positive solution of invariant subspace problem in
$H^{2}$, we also disprove the Riemann Hypothesis.

\section{\textbf{The Solution of the Transitive Algebra Problem in Hilbert
Spaces}}

The classical Burnside's theorem \cite{12, 20, B1905} states that any proper
subalgebra of the algebra of operators on a finite dimensional complex vector
space has a nontrivial invariant subspace. This is a much stronger assertion
than existence of invariant subspace for single operators: a singly generated
algebra is certainly proper, since it is commutative. In 1955, Kadison
\cite{14} asked whether the analogue of Burnside's theorem holds in the
infinite dimensional case (see Problem \ref{P3}). Beginning with the
fundamental paper \cite{A1967} of Arveson, many partial results are known on
Problem \ref{P3} in Hilbert spaces see for example, Radjavi and Rosenthal
\cite{28}, Lomonosov \cite{19, 20}, Shulman \cite{32}, Kissin, Shulman and
Turovskii \cite{16}, Mustafaev \cite{22}, Turovskii \cite{36, 37} and Kissin
\cite{15}.

In the present section, we prove that the answer to Problem \ref{P3} is
affirmative (see Theorem \ref{T1} below). Our proof is based on a simple
reproducing kernels argument for operators on the classical Hardy space
$H^{2}\left(  \mathbb{D}\right)  ;$ it is classical that every infinite
dimensional separable complex Hilbert space $H$ is isometrically isomorphic to
$\ell^{2},$ and hence to $H^{2}\left(  \mathbb{D}\right)  .$

Recall that the Hardy space $H^{2}=H^{2}\left(  \mathbb{D}\right)  $ consists
of all analytic functions $f\left(  z\right)  =\underset{n=0}{\overset{\infty
}{\sum}}\widehat{f}\left(  n\right)  z^{n\text{ \ }}$on $\mathbb{D}$ with the
sequence of Taylor coefficients $\left\{  \widehat{f}\left(  n\right)
\right\}  _{n\geq0}$ in $\ell^{2}.$ Equivalently, $H^{2}$ is the space of
analytic functions $f$ with finite integral means%
\[
\left\Vert f\right\Vert _{2}=\underset{0\leq r<1}{\sup}\left(  \frac{1}{2\pi
}\underset{0}{\overset{2\pi}{\int}}\left\vert f\left(  re^{it}\right)
\right\vert ^{2}dt\right)  ^{1/2}<+\infty.
\]

It is easy to show that $\left\Vert f\right\Vert _{2}=\left(  \underset
{n=0}{\overset{\infty}{\sum}}\left\vert \widehat{f}\left(  n\right)
\right\vert ^{2}\right)  ^{1/2}.$ $H^{\infty}=H^{\infty}\left(  \mathbb{D}%
\right)  $ is the Banach space of all bounded analytic functions $f$ on
$\mathbb{D}$ such that $\left\Vert f\right\Vert _{\infty}:=\underset
{z\in\mathbb{D}}{\sup}\left\vert f\left(  z\right)  \right\vert <+\infty.$

It follows from the inequality%
\[
\left\vert f\left(  z\right)  \right\vert \leq\left\Vert f\right\Vert
_{2}\left(  1-\left\vert z\right\vert ^{2}\right)  ^{-1/2},
\]
valid for $f\in H^{2}$ and $z\in\mathbb{D}$, that $H^{2}$ is a reproducing
kernel Hilbert space on $\mathbb{D}$. The set $\left\{  z^{n}:n\geq0\right\}
$ is an orthonormal basis for $H^{2}$, and hence its reproducing kernel has
the form
\[
k_{\lambda}\left(  z\right)  =\underset{n=0}{\overset{\infty}{\sum}}%
\overline{\lambda^{n}}z^{n}=\frac{1}{1-\overline{\lambda}z}(z,\lambda
\in\mathbb{D}).
\]

The function $\widehat{k}_{\lambda}\left(  z\right)  :=\frac{k_{\lambda
}\left(  z\right)  }{\left\Vert k_{\lambda}\right\Vert _{2}}$ is called the
normalized reproducing kernel at $\lambda.$ It is easy to see that
$\widehat{k}_{\lambda}\left(  z\right)  \longrightarrow0$ weakly as
$\lambda\longrightarrow\partial\mathbb{D}$, that is the Hardy space $H^{2}$ is
a standard reproducing kernel Hilbert space in sense of Nordgren and Rosenthal
\cite{24}. Indeed, note that $\left\langle f,\widehat{k}_{\lambda
}\right\rangle =\sqrt{1-\left\vert \lambda\right\vert ^{2}}f(\lambda)$ for
$f\in H^{2},$ and this obviously approaches $0$ for $f\in H^{\infty},$ and
hence for all $f\in H^{2}$ whenever $\left\vert \lambda\right\vert
\longrightarrow1^{\text{-}}.$ (For more information about reproducing kernels,
see, Aronszajn \cite{3}).

The following simplest lemma is surprisingly decisive in the proof of our main results.

\begin{lemma}
\label{L1}Let $\mu\in\mathbb{D}$ be fixed. For any operator $T$ in
$\mathcal{B}\left(  H^{2}\right)  $ we have
\[
\underset{\lambda\longrightarrow\partial\mathbb{D}}{\lim}\left(  T\widehat
{k}_{\lambda}\right)  \left(  \mu\right)  =0.
\]

\end{lemma}

\begin{proof}
The proof is immediate from the standardness property of $H^{2}.$ In fact,
since $\widehat{k}_{\lambda}\longrightarrow0$ weakly as $\lambda
\longrightarrow\partial\mathbb{D}$, we have:%
\[
\left(  T\widehat{k}_{\lambda}\right)  \left(  \mu\right)  =\left\langle
T\widehat{k}_{\lambda},k_{\mu}\right\rangle =\left\langle \widehat{k}%
_{\lambda},T^{\ast}k_{\mu}\right\rangle \longrightarrow0
\]
as $\lambda\longrightarrow\partial\mathbb{D}$. This proves the lemma.
\end{proof}

It can be given the following interpretation to the statement of Lemma
\ref{L1}. For any operator $T\in\mathcal{B}\left(  H^{2}\right)  $ we define
its "two-variables" Berezin symbol $\widetilde{T}_{tv}$ by the formula
\[
\widetilde{T}_{tv}\left(  \lambda,\mu\right)  :=\left\langle T\widehat
{k}_{\lambda},\widehat{k}_{\mu}\right\rangle \left(  \lambda,\mu\in
\mathbb{D}\right)  .
\]
For $\mu=\lambda,$ we have the usual Berezin symbol $\widetilde{T}$ of
operator $T:$%
\[
\widetilde{T}\left(  \lambda\right)  :=\left\langle T\widehat{k}_{\lambda
},\widehat{k}_{\lambda}\right\rangle \left(  \lambda\in\mathbb{D}\right)  ,
\]
which originally introduced by Berezin \cite{B1, B2}. It is trivial from the
Cauchy-Schwarz inequality that%
\begin{align*}
\underset{\lambda\in\mathbb{D}}{\sup}\left\vert \widetilde{T}\left(
\lambda\right)  \right\vert \text{ (Berezin number)}  &  \leq\underset
{\lambda,\mu\in\mathbb{D}}{\sup}\left\vert \widetilde{T}_{tv}\left(
\lambda,\mu\right)  \right\vert \text{ ("small" Berezin norm)}\\
&  \leq\underset{\lambda\in\mathbb{D}}{\sup}\left\Vert T\widehat{k}_{\lambda
}\right\Vert \text{ ("big" Berezin norm)}\leq\left\Vert T\right\Vert
\end{align*}
for every $T\in\mathcal{B}\left(  H^{2}\right)  .$

Since $\left\{  \widehat{k}_{\lambda}\right\}  $ weakly converges to zero
whenever $\lambda$ approaches to the boundary points $\zeta\in\partial
\mathbb{D}$, it is elementary that for any $\mu\in\mathbb{D}$ $\underset
{\lambda\longrightarrow\zeta}{\lim}\left\langle \widehat{k}_{\lambda},T^{\ast
}k_{\mu}\right\rangle =0$ if and only if $\underset{\lambda\longrightarrow
\zeta}{\lim}\left\langle \widehat{k}_{\lambda},T^{\ast}\widehat{k}_{\mu
}\right\rangle =0$ (since $\sqrt{1-\left\vert \mu\right\vert ^{2}}\neq0$ for
all $\mu\in\mathbb{D}$). This shows that the statement of Lemma 1 means that
\ $\underset{\lambda\longrightarrow\partial\mathbb{D}}{\lim}\left\langle
T\widehat{k}_{\lambda},\widehat{k}_{\mu}\right\rangle =0$ for all fixed
$\mu\in\mathbb{D}$, that is for every operator $T$ in $\mathcal{B}\left(
H^{2}\right)  $ its two-variables Berezin symbol $\widetilde{T}_{tv}\left(
\lambda,\mu\right)  $ "semi-vanishes"\ on the boundary $\partial\mathbb{D}$.

The main result of this section, is the following theorem which gives a
positive answer to the transitive algebra problem in Problem \ref{P3}, and
hence, it solves the Kadison question affirmatively.

\begin{theorem}
\label{T1}Every proper weakly closed unital subalgebra $\mathcal{A}$ of
algebra $\mathcal{B}\left(  H^{2}\right)  $ has a closed nontrivial invariant
subspace, i.e., $\mathcal{A}$ is nontransitive; equivalently, only transitive
subalgebra of $\mathcal{B}\left(  H^{2}\right)  $ is $\mathcal{B}\left(
H^{2}\right)  $ itself.
\end{theorem}

\begin{proof}
Let $\mathcal{A\subset B}\left(  H^{2}\right)  $ be a proper weakly closed
unital subalgebra. For every nonzero $g\in H^{2},$ we set
\[
E_{g}:=clos\mathcal{A}g:=clos\left\{  Ag:A\in\mathcal{A}\right\}  .
\]
Clearly, $E_{g}\neq\left\{  0\right\}  $ and $AE_{g}\subset E_{g}$ for all
$A\in\mathcal{A}$. We will prove that there exists $g_{0}\in H^{2}%
\backslash\left\{  0\right\}  $ such that $E_{g_{0}}$ is not dense in $H^{2}$,
i.e., $E_{g_{0}}\neq H^{2}$. For this aim, suppose in contrary that $E_{g}=$
$H^{2}$ for all nonzero $g\in H^{2}$. Then, in particular, $E_{\widehat
{k}_{\lambda}}=H^{2}$ for all $\lambda\in\mathbb{D}$. Let $f\in H^{2}$ be an
arbitrary nonzero function. Then for any $\varepsilon>0$ there exists
$A_{\varepsilon,\lambda,f}\in\mathcal{A}$ such that $\left\Vert
f-A_{\varepsilon,\lambda,f}\widehat{k}_{\lambda})\right\Vert _{2}%
<\varepsilon.$ Hence%
\[
\left\vert f\left(  0\right)  -(A_{\varepsilon,\lambda,f}\widehat{k}_{\lambda
})\left(  0\right)  \right\vert \leq\left\Vert f-A_{\varepsilon,\lambda
,f}\widehat{k}_{\lambda}\right\Vert _{2}<\varepsilon
\]
for every $\lambda\in\mathbb{D}$, that is
\[
\left\vert f\left(  0\right)  -(A_{\varepsilon,\lambda,f}\widehat{k}_{\lambda
})\left(  0\right)  \right\vert <\varepsilon\text{ }\left(  \lambda
\in\mathbb{D}\right)  .
\]
From this, for any $B\in\mathcal{A}$ we have
\begin{align*}
\left\vert f\left(  0\right)  \right\vert  &  <\left\vert (A_{\varepsilon
,\lambda,f}\widehat{k}_{\lambda})\left(  0\right)  \right\vert +\varepsilon\\
&  \leq\left\vert (A_{\varepsilon,\lambda,f}\widehat{k}_{\lambda})\left(
0\right)  -(B\widehat{k}_{\lambda})\left(  0\right)  \right\vert +\left\vert
(B\widehat{k}_{\lambda})\left(  0\right)  \right\vert +\varepsilon\\
&  \leq\left\Vert A_{\varepsilon,\lambda,f}\widehat{k}_{\lambda}-B\widehat
{k}_{\lambda}\right\Vert _{2}+\left\vert (B\widehat{k}_{\lambda})\left(
0\right)  \right\vert +\varepsilon\\
&  \leq\left\Vert A_{\varepsilon,\lambda,f}-B\right\Vert +\left\vert
(B\widehat{k}_{\lambda})\left(  0\right)  \right\vert +\varepsilon,
\end{align*}
hence by virtue of Lemma \ref{L1}, we deduce that for any $\varepsilon>0$
there exists $\delta_{\varepsilon}>0$ such that $\left\vert (B\widehat
{k}_{\lambda})\left(  0\right)  \right\vert <\varepsilon$ for all $\lambda
\in\mathbb{D}$ with $\left\vert \lambda-1\right\vert <\delta_{\varepsilon}$,
therefore we have that
\[
\left\vert f\left(  0\right)  \right\vert \leq\left\Vert A_{\varepsilon
,\lambda,f}-B\right\Vert +2\varepsilon
\]
for all $\lambda\in\mathbb{D}$ such that $\left\vert \lambda-1\right\vert
<\delta_{\varepsilon}.$ Consequently,%
\[
\left\vert f\left(  0\right)  \right\vert -2\varepsilon\leq\underset
{B\in\mathcal{A}}{\inf\text{ }}\text{ }\left\Vert A_{\varepsilon,\lambda
,f}-B\right\Vert =\mathrm{dist}(A_{\varepsilon,\lambda,f},\mathcal{A})=0,
\]
thus $\left\vert f\left(  0\right)  \right\vert \leq2\varepsilon$ for any
$\varepsilon>0,$ which is impossible since $f$ is arbitrary. The theorem is proven.
\end{proof}

Since $\left\{  T\right\}  ^{\prime}$ is a proper weakly closed unital algebra
of $\mathcal{B}\left(  H^{2}\right)  $ for any nonscalar operator
$T\in\mathcal{B}\left(  H^{2}\right)  ,$ the following corollary is immediate
from Theorem \ref{T1}. This solves the hyperinvariant subspace problem in
Hilbert spaces affirmatively (see Problem \ref{P2}).

\begin{corollary}
\label{C1}Every nonscalar bounded linear operator on the Hardy space $H^{2}$
has a closed nontrivial hyperinvariant subspace.
\end{corollary}

Since every hyperinvariant subspace is an invariant subspace of operator
$T\in\mathcal{B}\left(  H^{2}\right)  $, the following corollary gives a
positive answer to the invariant subspace problem in $H^{2},$ and hence, in
any infinite dimensional separable complex Hilbert space $H$ (see Problem
\ref{P1}).

\begin{corollary}
\label{C2}Every bounded linear operator on the Hardy space $H^{2}$ has a
closed nontrivial invariant subspace.
\end{corollary}

\section{\textbf{The Solution of the Riemann Hypothesis}}

This section is mainly motivated with the recent works of S. W. Noor \cite{1},
J. Manzur, W. Noor, C. F. Santos \cite{MNS2022, 2}, and A. Ghosh, Y.
Kremnizer, S. W. Noor \cite{4}. The goal of this section is to prove that the
positive solution of the invariant subspace problem (see Corollary \ref{C2} in
Section 2) in the Hardy space $H^{2}$ implies a negative answer to the
celebrated Riemann Hypothesis.

The Riemann Hypothesis says that all the nontrivial zeros of the Riemann
$\zeta-$function $\zeta\left(  z\right)  :=\underset{n=1}{\overset{\infty
}{\sum}}\frac{1}{n^{z}\text{ }}$ lie on the critical line $\operatorname{Re}%
(z)=\frac{1}{2}$ of the complex plane $\mathbb{C}$. Notice that the trivial
zeros of the $\zeta$-function are negative integers $-2k$, $k\geq1.$

Our proof of Riemann Hypothesis based on Corollary \ref{C2}, which is already
proved in Section 2, some important results of the paper \cite{2} related with
some semigroup of weighted composition operators on $H^{2}$ and Nordgren and
Rosenthal's standardness property of some subspaces in $H^{2}.$

In \cite{1}, the author defined the following subspace, namely, let
$\mathcal{N}$ denote the linear \textrm{span }(i.e., , linear hull) of the
functions%
\[
h_{k}\left(  z\right)  :=\frac{1}{1-z}\log\left(  \frac{1+z+\cdots+z^{k-1}}%
{k}\right)  ,\text{ }k\geq2,
\]
which all belong to the Hardy space $H^{2}\left(  \mathbb{D}\right)  $ and all
are outer functions see \cite[Lemma 7]{1} and \cite[Corollary 14]{MNS2022} (it
is also shown in \cite{1} that $\mathcal{N}\in\mathrm{Lat}(W),$ i.e.,
$W_{n}\mathcal{N\subset N}$ for all $n\geq2$ and Riemann Hypothesis holds if
and only if $\mathcal{N}$ is dense in $H^{2}$, i.e., $\overline{\mathcal{N}%
}=H^{2}$). Also in \cite{1}, a multiplicative semigroup of weighted
composition operators $W:=\left(  W_{n}\right)  _{n\in\mathbf{N}}$ on $H^{2}$
is introduced:%
\[
W_{n}f:=\left(  1+z+\cdots+z^{n-1}\right)  f\left(  z^{n}\right)
=\frac{1-z^{n}}{1-z}f\left(  z^{n}\right)  .
\]
Each $W_{n}$ is bounded on $H^{2}$, $W_{1}=I$ (identity operator on $H^{2}$)
and $W_{m}W_{n}=W_{mn}$ for each $m,n\geq1.$

The following is proved in \cite[Proposition 4]{2}:

\begin{proposition}
\label{O1}Each $\frac{W_{n}}{\sqrt{n}}$ for $n\geq2$ is a shift with infinite
multiplicity and hence $W_{n}^{\ast}$ is universal in the sense of Rota
\cite{6}.
\end{proposition}

A proper invariant subspace $E$ for an operator $T$ is called maximal if it is
not contained in any other proper invariant subspace for $T.$ In this case
$E^{\perp}$ is a minimal invariant subspace for $T^{\ast}.$ Hence the
invariant subspace problem may be reformulated in terms of $W_{n}$ (see
\cite[Corollary 6]{2}):

\begin{proposition}
\label{O2}For any $n\geq2,$ every maximal invariant subspace for $W_{n}$ has
co-dimension one if and only if the invariant subspace problem has a positive solution.
\end{proposition}

In \cite{2}, the authors notice that studying the invariant subspaces for $W$
may shed light on both the Riemann Hypothesis and the Invariant Subspace
Problem. Here, in this section, we will partially confirm this point of view.

Define the manifolds (see \cite[p. 5]{2}):%
\[
\mathcal{M}:=\mathrm{span}\left\{  h_{k}-h_{\ell}:k,\ell\geq2\right\}
\]
and
\[
\mathcal{M}_{d}:=\mathrm{span}\left\{  h_{k}-h_{\ell}:k,\ell\in d\mathbf{N}%
\right\}
\]
for $d\in\mathbf{N}$ whose closures belong to $\mathrm{Lat}(W)$ due to the
identity $W_{n}h_{k}=h_{nk}-h_{n}$ (see \cite{2}). It is clear that
$\mathcal{M}_{d}\subset\mathcal{M\subset N}$, and $d_{1}$ divides $d_{2}$ if
and only if $\mathcal{M}_{d_{2}}\subset\mathcal{M}_{d_{1}.}$ It follows that
$\left(  \mathcal{M}_{d}\right)  _{d\in\mathbf{N}}$ is a sublattice of
$\mathrm{Lat}(W)$ which is isomorphic to $\mathbf{N}$ with respect to
division, and $\mathcal{M}_{p}$ for any prime $p$ is a maximal element in
$\left(  \mathcal{M}_{d}\right)  _{d\in\mathbf{N}}.$

The following important connection between the maximality property of the
subspace $\overline{\mathcal{M}}$ and the Riemann Hypothesis is proved in
\cite[Theorem 9]{2}. Let $\vee_{n}E_{n}$ denote the smallest closed subspace
containing the sets $E_{n}.$

\begin{theorem}
\label{T2}(\cite{2}). The closure of $\mathcal{M}$ is a proper element in
$\mathrm{Lat}(W)$ and equals $\underset{d=2}{\overset{\infty}{\vee}%
}\mathcal{M}_{d}.$ The closure of $\mathcal{M}$ is maximal in $\mathrm{Lat}%
(W)$ if and only if the Riemann Hypothesis is true. In this case we have
co-$dim(\mathcal{M})=1.$
\end{theorem}

It is also proved in \cite{2} that%
\[
\underset{n=k+1}{\overset{\infty}{\cap}}\ker(W_{n}^{\ast})=\underset{1\leq
n\leq k}{\mathrm{span}}\left\{  1-z^{n}\right\}  ,\text{ }\forall k\geq1.
\]
\newline In particular, for $k=1$ we have that%
\begin{equation}
\underset{n=2}{\overset{\infty}{\cap}}\ker(W_{n}^{\ast})=\left\{  1-z\right\}
.\label{1}%
\end{equation}
On the other hand, it is shown in \cite{2} that $\overline{\mathcal{M}%
}=\underset{n=2}{\overset{\infty}{\vee}}W_{n}\left(  \mathcal{N}\right)  .$
Then we have by virtue of (\ref{1}) that%
\[
\overline{\mathcal{M}}=\underset{n=2}{\overset{\infty}{\vee}}W_{n}\left(
\mathcal{N}\right)  \subset\underset{n=2}{\overset{\infty}{\vee}}I_{m}%
W_{n}=\left(  \underset{n=2}{\overset{\infty}{\cap}}\ker W_{n}^{\ast}\right)
^{1}=\left\{  1-z\right\}  ^{\perp},
\]
which implies that%
\begin{equation}
\left\{  1-z\right\}  \subset\overline{\mathcal{M}}^{^{\perp}}.\label{2}%
\end{equation}
Now it follows from (\ref{2}) that $\dim\overline{\mathcal{M}}^{^{\perp}}%
\geq1$ i.e., co$-\dim\left(  \overline{\mathcal{M}}\right)  \geq1.$ Let us
show that actually $\dim\overline{\mathcal{M}}^{\perp}>1.$ Indeed, if
$\dim\overline{\mathcal{M}}^{\perp}=1,$ then inclusion (\ref{2}) shows that
reproducing kernel $k_{\overline{\mathcal{M}}^{^{\perp}},\lambda}(z)$ of the
subspace $\overline{\mathcal{M}}^{\perp}$ has the form:%
\[
k_{\overline{\mathcal{M}}^{^{\perp}},\lambda}(z)=\frac{\left(  \overline
{1-\lambda}\right)  \left(  1-z\right)  }{2},\text{ }\lambda\in\mathbb{D},
\]
and hence%
\begin{align*}
k_{\overline{\mathcal{M}},\lambda}(z) &  =k_{\lambda}\left(  z\right)
-k_{\overline{\mathcal{M}}^{^{\perp}},\lambda}(z)=\\
&  =\frac{1}{1-\overline{\lambda}z}-\frac{\left(  \overline{1-\lambda}\right)
\left(  1-z\right)  }{2}=\\
&  =\frac{2-\left(  1-\overline{\lambda}\right)  \left(  1-z\right)  \left(
1-\overline{\lambda}z\right)  }{2\left(  1-\overline{\lambda}z\right)  },
\end{align*}
thus%
\begin{equation}
k_{\overline{\mathcal{M}},\lambda}(z)=\frac{2-\left(  1-\overline{\lambda
}\right)  \left(  1-z\right)  \left(  1-\overline{\lambda}z\right)  }{2\left(
1-\overline{\lambda}z\right)  },\text{ }\lambda\in\mathbb{D}.\label{3}%
\end{equation}
It is easy to see from (\ref{3}) that
\begin{equation}
\left\Vert k_{\overline{\mathcal{M}},\lambda}\right\Vert =\frac{\sqrt{2}}%
{2}\frac{\sqrt{\left\vert 1+\lambda\right\vert ^{2}\left(  1-\left\vert
\lambda\right\vert ^{2}\right)  +2\left\vert \lambda\right\vert ^{4}}}{\left(
1-\left\vert \lambda\right\vert ^{2}\right)  ^{1/2}}.\label{4}%
\end{equation}
In fact, we have:%
\begin{align*}
\left\Vert k_{\overline{\mathcal{M}},\lambda}\right\Vert _{2} &  =\frac{1}%
{2}\left\Vert \frac{2}{1-\overline{\lambda}z}-\left(  1-\overline{\lambda
}\right)  \left(  1-z\right)  \right\Vert _{2}=\\
&  =\frac{1}{2}\left\Vert \left[  2\underset{k=0}{\overset{\infty}{\sum}%
}\overline{\lambda}^{k}z^{k}-\left(  1-z-\overline{\lambda}+\overline{\lambda
}z\right)  \right]  \right\Vert _{2}\\
&  =\frac{1}{2}\left\Vert \left(  1+\overline{\lambda}\right)  1+\left(
\overline{\lambda}+1\right)  z+2\overline{\lambda}^{2}z^{2}+2\overline
{\lambda}^{3}z^{3}+\cdots\right\Vert _{2}\\
&  =\frac{1}{2}\left\{  2\left\vert 1+\lambda\right\vert ^{2}+4\left[
\left\vert \lambda\right\vert ^{4}+\left\vert \lambda\right\vert
^{6}+\left\vert \lambda\right\vert ^{8}+\cdots\right]  \right\}  ^{1/2}\\
&  =\frac{\sqrt{2}}{2}\left[  \left\vert 1+\lambda\right\vert ^{2}+2\left[
\left\vert \lambda\right\vert ^{4}\left(  1+\left\vert \lambda\right\vert
^{2}+\left\vert \lambda\right\vert ^{4}+\cdots\right)  \right]  \right]
^{1/2}\\
&  =\frac{\sqrt{2}}{2}\left[  \left\vert 1+\lambda\right\vert ^{2}%
+\frac{2\left\vert \lambda\right\vert ^{4}}{1-\left\vert \lambda\right\vert
^{2}}\right]  ^{1/2}=\frac{\sqrt{2}}{2}\frac{\sqrt{\left\vert 1+\lambda
\right\vert ^{2}\left(  1-\left\vert \lambda\right\vert ^{2}\right)
+2\left\vert \lambda\right\vert ^{4}}}{\sqrt{1-\left\vert \lambda\right\vert
^{2}}},
\end{align*}
hence%
\[
\left\Vert k_{\overline{\mathcal{M}},\lambda}\right\Vert _{2}=\frac{\sqrt{2}%
}{2}\frac{\sqrt{\left\vert 1+\lambda\right\vert ^{2}\left(  1-\left\vert
\lambda\right\vert ^{2}\right)  +2\left\vert \lambda\right\vert ^{4}}}%
{\sqrt{1-\left\vert \lambda\right\vert ^{2}}}.
\]
Thus%
\begin{equation}
\widehat{k}_{\overline{\mathcal{M}},\lambda}(z)=\frac{2\left(  \overline
{1-\lambda}\right)  \left(  1-z\right)  \left(  1-\overline{\lambda}z\right)
}{\sqrt{2}\left(  1-\overline{\lambda}z\right)  }\frac{\sqrt{1-\left\vert
\lambda\right\vert ^{2}}}{\sqrt{\left\vert 1+\lambda\right\vert ^{2}\left(
1-\left\vert \lambda\right\vert ^{2}\right)  +2\left\vert \lambda\right\vert
^{4}}}.\label{5}%
\end{equation}
Since $\mathcal{M}=\mathrm{span}\left\{  h_{k}-h_{\ell}:k,\ell\geq2\right\}
,$ where%
\[
h_{k}\left(  z\right)  =\frac{1}{1-z}\log\frac{1+z+\cdots+z^{k-1}}{k},\text{
}k\geq2,
\]
by using representation (\ref{5}) we have%
\begin{align}
&  \left\langle h_{k}\left(  z\right)  -h_{\ell}\left(  z\right)  ,\widehat
{k}_{\overline{\mathcal{M}},\lambda}(z)\right\rangle \nonumber\\
&  =\sqrt{2}\left(  \frac{1}{1-\lambda}\log\frac{\ell\left(  1+\lambda
+\cdots+\lambda^{k-1}\right)  }{k\left(  1+\lambda+\cdots+\lambda^{\ell
-1}\right)  }\right)  \frac{\sqrt{1-\left\vert \lambda\right\vert ^{2}}}%
{\sqrt{\left\vert 1+\lambda\right\vert ^{2}\left(  1-\left\vert \lambda
\right\vert ^{2}\right)  +2\left\vert \lambda\right\vert ^{4}}}\label{6}%
\end{align}
for all $\lambda\in\mathbb{D}$ and all $k,\ell\geq2.$ It is clear that
\[
\underset{\lambda\longrightarrow\zeta}{\lim}\frac{1}{1-\lambda}\log\frac
{\ell\left(  1+\lambda+\cdots+\lambda^{k-1}\right)  }{k\left(  1+\lambda
+\cdots+\lambda^{\ell-1}\right)  }\frac{\sqrt{1-\left\vert \lambda\right\vert
^{2}}}{\sqrt{\left\vert 1+\lambda\right\vert ^{2}\left(  1-\left\vert
\lambda\right\vert ^{2}\right)  +2\left\vert \lambda\right\vert ^{4}}}=0
\]
for all $\zeta\in\partial\mathbb{D\setminus}\left\{  1\right\}  .$ For
$\zeta=1,$ it is not difficult to see from the L'H\^{o}spital Rule that%
\[
\underset{\lambda\longrightarrow1}{\lim}\frac{1}{1-\lambda}\log\frac
{\ell\left(  1+\lambda+\cdots+\lambda^{k-1}\right)  }{k\left(  1+\lambda
+\cdots+\lambda^{\ell-1}\right)  }%
\]
exists and is a finite nonzero number, and therefore%
\begin{equation}
\underset{\lambda\longrightarrow1}{\lim}\frac{1}{1-\lambda}\log\frac
{\ell\left(  1+\lambda+\cdots+\lambda^{k-1}\right)  }{k\left(  1+\lambda
+\cdots+\lambda^{\ell-1}\right)  }\frac{\sqrt{1-\left\vert \lambda\right\vert
^{2}}}{\sqrt{\left\vert 1+\lambda\right\vert ^{2}\left(  1-\left\vert
\lambda\right\vert ^{2}\right)  +2\left\vert \lambda\right\vert ^{4}}%
}=0.\label{7}%
\end{equation}
So, it follows from (\ref{6}) and (\ref{7}) that $\left\langle f,\widehat
{k}_{\overline{\mathcal{M}},\lambda}\right\rangle \longrightarrow0$ whenever
$\lambda\longrightarrow\zeta\in\partial\mathbb{D}$ for all $f\in
\overline{\mathcal{M},}$ that means that $\overline{\mathcal{M}}$ is a
standard reproducing kernel Hilbert space.

Thus we have from (\ref{2}) and (\ref{5}) that $\left\langle
f,1-z\right\rangle =0$ for all $f\in$ $\overline{\mathcal{M}}$. Then
\[
0=\left\langle f,1-z\right\rangle =\left\langle f,1-z-\widehat{k}%
_{\overline{\mathcal{M}},\lambda}\right\rangle +\left\langle f,\widehat
{k}_{\overline{\mathcal{M}},\lambda}\right\rangle ,
\]
or equivalently
\[
\left\langle f,\widehat{k}_{\overline{\mathcal{M}},\lambda}-\left(
1-z\right)  \right\rangle =\left\langle f,\widehat{k}_{\overline{\mathcal{M}%
},\lambda}\right\rangle ,
\]
since $\left\langle f,\widehat{k}_{\overline{\mathcal{M}},\lambda
}\right\rangle \longrightarrow0$ as $\lambda\longrightarrow\zeta\in
\partial\mathbb{D},$ we decude that
\[
\underset{\lambda\longrightarrow\zeta\in\partial\mathbb{D}}{\lim}\left\langle
f,\widehat{k}_{\overline{\mathcal{M}},\lambda}-\left(  1-z\right)
\right\rangle =0,
\]
which means that $\widehat{k}_{\overline{\mathcal{M}},\lambda}\overset
{weakly}{\longrightarrow}1-z$ as $\lambda\longrightarrow\partial\mathbb{D}.$
This contradicts to $1-z\neq0.$ So, $\dim\left(  \overline{\mathcal{M}}^{\bot
}\right)  >1.$ Hence the following lemma is proven.

\begin{lemma}
\label{L2}co-$\dim\left(  \overline{\mathcal{M}}\right)  >1.$
\end{lemma}

Now we are ready to state the main result of the present section, which
disproves the Riemann Hypothesis.

\begin{theorem}
\label{T3}The Riemann Hypothesis is not true.
\end{theorem}

\begin{proof}
It follows from Corollary \ref{C2} in Section 2 and Proposition \ref{O2} that
every maximal invariant subspace of operators $W_{n},$ $n\geq2,$ has
co-dimension one. However, according to Lemma \ref{L2}, co-$dim(\overline
{\mathcal{M}})>1$, and hence the subspace $\overline{\mathcal{M}}$ can not be
a maximal invariant subspace for $W_{n},$ $n\geq2,$ thus by virtue of Manzur,
Noor and Santos theorem (see Theorem \ref{T2}), we deduce that the Riemann
Hypothesis is not true. The theorem is proven.
\end{proof}

\end{document}